\documentclass[12pt]{amsart}
\usepackage{graphicx}
\usepackage{amsmath}
\usepackage{amssymb}
\usepackage{tikz-cd}

\newtheorem{theorem}{Theorem}[section]
\newtheorem{lemma}[theorem]{Lemma}
\newtheorem{corollary}[theorem]{Corollary}
\newtheorem{proposition}[theorem]{Proposition}

\theoremstyle{definition}
\newtheorem{definition}[theorem]{Definition}

\newtheorem{example}[theorem]{Example}

\numberwithin{equation}{section}
\numberwithin{equation}{section}

\begin{document}

\title[Prime subcomplexes]{Prime subcomplexes}

\author[I. Akray]{Ismael Akray}
\address{Mathematics Department, Faculty of Science\\
	Soran University\\ 44008, Soran, Erbil
	Kurdistan Region, Iraq}

\email {akray.ismael@gmail.com, ismael.akray@soran.edu.iq}


\begin{abstract}

In 1859, Ernst E. Kummer \cite{erns} introduced the concept prime ideal and after that in 1983 this concept was generalized to prime submodules by R. McCasland \cite{mcca}. In this article, by considering complexes as a generalization of modules, I introduce the concept prime subcomplex and study some properties analogous to that of prime submodules.

\end{abstract}

\maketitle

\section{Introduction}

 Throughout this article all rings are assumed to be commutative with identity and all modules will be unitary. The symbol $\sqrt{P}$ and $Ann(P)$ indicates to the radical and the annihilator of $P$.  E. E. Kummer \cite{erns} introduced the notion prime ideal and then many authors discussed their properties like Krull, Fitting and McCoy in \cite{krul, fitt, mcco}. A proper ideal $P$ of a ring $R$ is said to be prime if $ab \in P$ for $a,b \in R$ implies $a \in P$ or $b \in P$. In other words, an ideal $P$ is prime if and only if $R/P$ is an integral domain. E. Noether \cite{noet} introduced primary ideal in 1921, where a proper ideal $P$ of a ring $R$ is said to be primary if $ab \in P$ for $a,b \in R$ implies $a \in P$ or $b \in \sqrt{P}$. R. McCasland \cite{mcca} in 1983 generalized this concept into module theory and introduced prime submodules where a proper submodule $N$ of an $R-$module $M$ is prime if $rm \in N$ for $r \in R$ and $m \in M$ implies $m \in N$ or $r \in (N:_R M) = Ann_R(M/N)$. Equivalently, $N$ is prime submodule if it is a primary submodule whose radical is identical with $(N:M)$, where a proper submodule $K$ of $M$ is primary (see \cite{noet2}) if $rm \in K$ for $r \in R$ and $m \in M$ implies $m \in K$ or $r \in \sqrt{(N:M)}$. In the theory of rings, prime ideals plays an important role everywhere and so does prime submodules in the theory of modules. In particular, primary submodules are the major characters in Noetherian modules.

\vspace{.3cm}
 H. Cartan \cite{cart} introduced the notion $R-$complex in a siminar in 1959. An $R-$complex $C=(C, d)_{i \in I}$ is a sequence

{$$  \cdots \stackrel{d_{n+2}}\rightarrow C_{n+1} \stackrel{d_{n+1}}\rightarrow C_n \stackrel{d_n}\rightarrow C_{n-1} \stackrel{d_{n-1}}\rightarrow \cdots $$}

 of $R-$modules $C_n$ and $R-$morphisms $d_n$ such that $d_n d_{n+1}=0$ for all $n \in I$ \cite{rotm}.
 Also, a complex $S=(S_i, d_i)$ is called a subcomplex of an $R-$complex $C=(C_i,d_i)$ briefly $S \leq C$ if $S_i \subseteq C_i$ and $d_{i+1}(S_{i+1}) \subseteq S_i$ for all $i \in I$. Furthermore, if $S_i \neq C_i$ for some $i \in I$, then $S$ is called a proper subcomplex of $C$. We denote the set of all indices $i \in I$ in which $S_i \neq C_i$ by $P_I$. Define the set of zero divisors $Z_R (C)$ (briefly $Z (C)$) on an $R-$complex $C$ to be the intersection of all zero divisors sets on the components of $C$, that is, $Z_R(C)= \cap_{i\in I} Z_R(C_i) = \{r \in R: rm_i=0, \hspace{.2cm} \text{for some} \hspace{.2cm} m_i \in C_i, \forall i \in I \}$. Moreover, the annihilator of an $R-$complex $C$ is $Ann_R(C)=\{r \in R : rC_i=0, \hspace{.2cm} \text{for all} \hspace{.2cm}  i\in I \}$. So, the zero subcomplex of an $R-$complex $C$ is prime subcomplex if $Ann(C)=Z(C)$. Furthermore, we define the residual of a subcomplex $S$ of a complex $C$ by $(S:C) = \cap_{i\in I} (S_i:C_i)= \{r \in R: rC_i \subseteq S_i, \forall i \in I \}$. Every $R-$module $M$ can be considered as a complex with $C_0=M$ and $C_i=0$ for all $(i \neq 0) \in I$ and all morphisms be zero. This motivate us to think about how can we introduce or generalize many notions in ring theory and module theory into $R-$complexes.

\vspace{.3cm}

 The purpose of this paper is to introduce and investigate the notion of primeness in a more wide category, the category of $R-$complexes. A proper subcomplex $S$ of an $R-$complex $C$ is said to be a prime subcomplex if $S_i$ is a prime submodule of $C_i$ for all $i \in P_I$. In other sense, A proper subcomplex $S$ of a complex $C$ is prime if for each $i \in I$, either $\frac{C_i}{S_i}$ is identically zero or a torsion-free $\frac{R}{P_i}$-module where $P_i= (S_i:C_i)$. In such case, we call $S$ a $P_i-$prime subcomplex. A proper subcomplex $S$ of an $R-$complex $C$ is primary if $S_i$ is primary for all $i \in P_I$, that is, for all $r \in R$, $i \in P_I$ and $m \in C_i$ with $rm \in S_i$ implies $m \in S_i$ or $r \in \sqrt{(S_i:C_i)}$. Equivalently, the zero-divisors of $\frac{C_i}{S_i}$ are nilpotent $\forall i \in P_I$. In this case we say $S$ is $P_i-$primary where $P_i= \sqrt{(S_i:C_i)}$. It is clear that a proper subcomplex $S$ of an $R-$complex $C$ is prime if and only if the chain map $f=<f_i>: \frac{C}{S} \rightarrow \frac{C}{S}$ defined by $f_i(m+S_i)= rm+S_i$, for $m \in C_i$ is either monic or nilpotent for all nonzero $r \in R$.

\vspace{.3cm}

 In section two, we establish the most fundamental theorems of prime submodules to prime subcomplexes. We first point out that the notion of prime subcomplex have many equivalents (Theorem \ref{equi}). We show that a subcomplex $S$ of an $R-$complex $C$ is prime if $(S_i:C_i)$ is a maximal ideal of $R$ for all $i \in P_I$ (Theorem \ref{p2}). Also, for any faithfully flat $R-$complex $F$, a subcomplex $S$ of a complex $C$ is prime if and only if $F \otimes S$ is a prime subcomplex of $F \otimes C$ (Theorem \ref{ff}). Finally, the prime avoidance theorem was generalized to prime subcomplexes (Theorem \ref{pri}).

\vspace{.3cm}

\section{Prime subcomplexes}
\hskip 0.6cm

In this section we introduce prime subcomplex and investigate some properties and results on such concept. We begin with their definition.
\begin{definition}
	 A proper subcomplex $S$ of an $R-$complex $C$ is said to be a prime subcomplex if $S_i$ is a prime submodule of $C_i$ for all $i \in P_I$.
\end{definition}

Let $\mathbb{N}$ be the set of natural numbers and $\mathbb{N}_0 = \mathbb{N} \cup \{0\}$. It is well-known that for any non-nilpotent element $a \in R$, the localization of a ring $R$ at the multiplicative subset $S= \{a^n: n \in \mathbb{N}_0 \}$ is $R_{(a)} =\{ \frac{r}{a^n} : r \in R, n \in \mathbb{N}_0 \} = R[\frac{1}{a}] = R[a, \frac{1}{a}]$, which is also called the Laurent polynomial ring. This localized ring has a principal prime ideal $<\frac{1}{a}> = \{ \frac{r}{a^n} : r \in R, n \in \mathbb{N}, a \nmid r \}$ and a primary ideal $<\frac{1}{a^2}>$. We denote such ideals by $P_a(R)$ and $(\frac{1}{a})  P_a(R)$ respectively. Now, we give an example on prime and primary subcomplexes in $\breve{C}$ech complexes \cite[5.1.6 Example, pp. 84]{brod}.


\begin{example}
	The $\breve{C}$ech complex of the ring of integers $\mathbb{Z}$ as a module over itself with respect to integers $3$, $5$ and $7$ is
	$$C^{\bullet}: 0 \to \mathbb{Z} \stackrel{d^0}\to \mathbb{Z}_{(3)} \oplus \mathbb{Z}_{(5)} \oplus \mathbb{Z}_{(7)} \stackrel{d^1}\to \mathbb{Z}_{(35)} \oplus \mathbb{Z}_{(21)} \oplus \mathbb{Z}_{(15)} \stackrel{d^2}\to \mathbb{Z}_{(105)} \to 0 $$
	
	where
	$$ d^0(n)= (\frac{n}{1}, \frac{n}{1}, \frac{n}{1})$$
	
	$$d^1((\frac{n_1}{3^{m_1}}, \frac{n_2}{5^{m_2}}, \frac{n_3}{7^{m_3}})) =	(\frac{5^{m_3}n_3}{35^{m_3}} - \frac{7^{m_2}n_2}{35^{m_2}}, \frac{3^{m_3}n_3}{21^{m_3}} - \frac{7^{m_1}n_1}{21^{m_1}},   \frac{3^{m_2}n_2}{15^{m_2}} - \frac{5^{m_1}n_1}{15^{m_1}}   ) $$
	
	$$ d^2 ((\frac{n_1}{35^{m_1}}, \frac{n_2}{21^{m_2}}, \frac{n_3}{15^{m_3}})) =  \frac{7^{m_3}n_3}{105^{m_3}} - \frac{5^{m_2}n_2}{105^{m_2}} + \frac{3^{m_1}n_1}{105^{m_1}}  $$
Then the subcomplex

$$S: 0 \to 0 \stackrel{\bar{d^0}}\to P_{3} \oplus \mathbb{Z}_{(5)} \oplus \mathbb{Z}_{(7)} \stackrel{\bar{d^1}}\to \mathbb{Z}_{(35)} \oplus \mathbb{Z}_{(21)} \oplus \mathbb{Z}_{(15)} \stackrel{\bar{d^2}}\to \mathbb{Z}_{(105)} \to 0$$	
	is a prime subcomplex of $C^{\bullet}$ while the subcomplex
	$$ T:  0 \to 0 \stackrel{\bar{d^0}}\to (\frac{1}{3})P_{3} \oplus \mathbb{Z}_{(5)} \oplus \mathbb{Z}_{(7)}  \stackrel{\bar{d^1}}\to \mathbb{Z}_{(35)} \oplus \mathbb{Z}_{(21)} \oplus \mathbb{Z}_{(15)} \stackrel{\bar{d^2}}\to \mathbb{Z}_{(105)} \to 0$$
	is a primary subcomplex of $C^{\bullet}$ not a prime.
\end{example}

\vspace{.3cm}

A proper subcomplex $S= (S_i, d_i)$ of a complex $C=(C_i, d_i)$ is said to be a direct summand of $C$ if either $S_i=C_i$ or $S_i$ is a direct summand of the $R-$module $C_i$ and denoted by $S \leq _{\oplus} C$. The torsion subcomplex $S=T(C)$ of an $R-$-complex $C$ is a subcomplex $S=\{S_i\}_{i\in I}$, where $S_i =T(C_i)$ is the torsion submodule of $C_i$, for all $i \in I$. We say that a complex $C$ is torsion if $C=T(C)$ and torsion-free if $T(C)=0$ is a zero complex.

\begin{proposition}
	Every direct summand subcomplex $S$ of a torsion-free complex $C$ is prime.

\end{proposition}

\begin{proof}
	Assume $S=(S_i,d_i)$ is a direct summand of a torsion-free complex $C=(C_i, d_i)$. Hence there exists a direct summand subcomplex $T$ such that $C=S \oplus T$, that means $C_i= S_i \oplus T_i$ for all $i$. Suppose $rm \in S_j$ for all $j \in P_I$, $r \in R$ and $m \in C_j= S_j \oplus T_j$. So, $m=a+b$ for $a \in S_j$, $b \in T_j$ and $rm - ra=rb \in S_j$ but $rb \in T_j$, hence $rb=0$ and torsion-freeness of $M_j$ gives us either $r=0$ or $b=0$. If $r=0$, then $rM_j=0 \subseteq S_j$, while $b=0$ implies that $m=a+0 \in S_j$. Therefore $S$ is a prime subcomplex of $C$.
\end{proof}

A subcomplex $S$ of a complex $C$ is pure if $rC_j \cap S_j= rS_j$, for $r \in R$ and $j\in P_I$. Now, we give a connection between primeness and pureness of subcomplexes in torsion-free complexes.

\begin{proposition}
	 A proper subcomplex $S$ of a torsion-free $R-$complex $C$ is prime with $(S_i:C_i)=0$, for all $i \in P_I$ if and only if $S$ is pure subcomplex of $C$.
\end{proposition}

\begin{proof}
	Let $S$ be a prime subcomplex of $C$ with $(S_i : C_i)=0$, for all $i\in P_I$. Take $r \in R$ and $m\in C_j$ for $j\in P_I$. Then primeness of $S_j$ implies $r \in (S_j:C_j)=0$ or $m \in S_j$. In both cases we will have $rS_j =S_j \cap rC_j$ and $S_j$ is pure submodule of $C_j$. So $S$ is pure subcomplex of $C$. Now, suppose $S$ is a pure subcomplex of $C$ and $rm \in S_j$ for $r \in R$ and $m\in C_j$ for all $j \in P_I$. From the identity $rC_j \cap S_j= rS_j$, we obtain $rm \in rS_j$, so there exists $n \in S_j$ with $rm=rn$ and that $r(m-n)=0$. Torsion-freeness of $C$ implies either $r=0$ or $m=n$, that is, $rC_j \subseteq S_j$ or $m \in S_j$. Thus $S_j$ is a prime submodule of $C_j$ and that $S$ is prime subcomplex of $C$.
\end{proof}

\begin{proposition}
	Suppose that $R$ is an integral domain and $C$ is an $R-$complex with $T(C) \neq C$. Then $T(C)$ is a prime subcomplex of $C$.
\end{proposition}

\begin{proof}
	Suppose $rm \in T(C_j)$ for nonzero $r \in R$ and $m \in C_j$ for $j \in I$ such that $T(C_j) \neq C_j$. As $T(C_j)$ is a torsion submodule of $C_j$ and $R$ is an integral domain, there exists a nonzero $s$ in $R$ such that $srm =0$ with $sr\neq 0$. Hence $m \in T(C_j)$ and the result follows.
\end{proof}

Let $f: R \rightarrow S$ be a ring homomorphism and $C$ be an $S-$complex. Then one can consider $C$ as an $R-$complex and the inverse image of each prime subcomplex of an $S-$complex $C$ is a prime subcomplex of $C$ as $R-$complex. For a primary subcomplex $S$ of $C$, $S$ is prime if and only if all proper ideals $(S_i:C_i)$ are prime $\forall i \in P_I$. Hence we have the following result.

\begin{proposition}
	Suppose $S$ is $P_i-$primary subcomplex of $C$ containing a $P_i-$prime subcomplex of $C$. Then $S$ is a $P_i-$prime.
\end{proposition}
\begin{proof}
	Consider a $P_i=(T_i:C_i)-$prime subcomplex $T$ of $C$ with $T \subset S$ and $S$ a $P_i=\sqrt{(S_i:C_i)}-$primary subcomplex. To prove $S$ is $P_i-$prime, we have to prove $(S_i:C_i)$ is prime foe all $i \in P_I$. For this purpose, let $abm \in S_i$, for all $m \in C_i$, $a,b \in R$ and $i \in P_I$. The primaryness of $S$ gives us $a \in \sqrt{(S_i:C_i)} $ or $b \in (S_i:C_i)$. If $a \in \sqrt{(S_i:C_i)}= (T_i:C_i)$, then $aC_i \subseteq T_i \subseteq S_i$ and so $a \in (S_i:C_i)$.
\end{proof}

Let $P$ be a prime ideal of $R$ and $S$ a subcomplex of $C$. The localization of $S$ at $P$ is $S_P=((S_i)_P, \bar{d}_i)$ and the saturation of $S$ at $P$ is the subcomplex $\c{S}_P (S)= (f^{-1} ((S_i)_P), \bar{d}_i)$, where $f=<f_i>$ is the natural chain map defined by $f_i:S_i \rightarrow (S_i)_P$, $f_i(x)=\frac{x}{1}$, for $x \in S_i$.

\begin{proposition}\label{p1}
	Suppose $S$ is a proper subcomplex of an $R-$complex $C$ such that $(S_i:C_i)=P_i$ is a prime ideal of $R$ for all $i \in P_I$ and set $P= <P_i>$. Then the saturation $\c{S}_P (S)= f^{-1}(S_P)$ of $S$ is a $P_i-$prime, where $$(S_P)_i= \begin{cases}
            {S_i}_{P_i}, & \mbox{if } P_i \neq R \\
            (C_i)_{P_i}, & \mbox{otherwise}.
          \end{cases}$$

and $f$ is the chain map $f_i: C_i \rightarrow (C_i)_{P_i}$, $f_i(m)= \frac{m}{1}$.
	
\end{proposition}
\begin{proof}
	Let $rm \in f^{-1}({S_i}_{P_i})$, for $r \in R$, $m \in C_i$ and $i \in P_I$. Then $\frac{rm}{1} \in {S_i}_{P_i}$ and so $trm \in S_i$ for $t \notin P$. If $r \notin P_i$, then $f_i(m)=\frac{m}{1}=\frac{trm}{tr} \in {S_i}_{P_i}$ and $m \in f^{-1}_i ((S_i)_{P_i})$. Therefore, $\c{S}_P (S)$ is a $P_i-$prime subcomplex of $C$.
\end{proof}

A free complex of finite rank is a complex whose terms are free $R-$modules of finite rank.

\begin{theorem}
	Let $S$ be a subcomplex of a free complex $F = (F_i, d_i)$ of finite rank. Then $S$ is a $P_i-$prime if and only if $S= \c{S}_P (T)$, for some subcomplex $T= (N_i+P_iF_i, d_i)$ of $F$, where $N=f^{-1}(A)$ for a direct summand $A$ of $F_{P}= ((F_i)_{P_i}, \bar{d}_i)$ and the canonical chain map $f : F \rightarrow F_{P_i}$, $f_i(x)=\frac{x}{1}$, $\forall x \in F_i$.
\end{theorem}

\begin{proof}
	Clearly $(N_i+P_iF_i:F_i)=P_i$ and so by Proposition \ref{p1}, $\c{S}_P (T)$ is a $P_i-$prime subcomplex of $F$. Now, let $S$ be a $P_i-$prime subcomplex of $F$. So, $P_iF_i \subset S_i$ and we have the composition of chain maps $F \stackrel f\to F_{P} \stackrel g\to \dfrac{F_P}{PF_P}$. Choose a basis $\{ e_1, \cdots e_{v_i} \}$ for $g_i({S_i}_{P_i})$ and a set of elements $m_1, \cdots , m_{v_i} \in {S_i}_{P_i} \subseteq (F_i)_{P_i}$ such that $g_j(m_i) = e_i$ for $i=1, \cdots , v_j$ and $j \in I$. Then one can check that $(S_i)_{P_i}= (m_1, \cdots , m_{v_i})+P_i {F_i}_{P_i}$. Since $P_i (F_i)_{P_i} \subseteq (S_i)_{P_i}$, we have an exact sequence $0 \to P_i (F_i)_{P_i} \to S_i \to \frac{S_i}{P_i (F_i)_{P_i}} \to 0$. Thus Nakayama lemma ensure that $<m_1, \cdots , m_{v_i}>$ is a direct summand of ${F_i}_{P_i}$. Put $N_i=f^{-1}_i (<m_1, \cdots , m_{v_i}>)$ and $N=(N_i, d_i)$. So, we have $S_i \subset N_i + F_i \subset \c{S}_P (N +P_iF) $ and $\c{S}_P (C+P_iF)$ is a $P_i-$prime subcomplex of $F$. By \cite[Corollary 3]{bour}, there exists a bijection correspondence between the $P-$prime subcomplexes of $F$ and those of $F_{P_i}$. Since $S_P=C_P+P_i F_P$, we get that $S=\c{S}_P (C+P_iF)$.

\end{proof}

From \cite[Page 169, Ex. 12]{bour2}, every primary ideal in a von Neumann regular ring is a prime. So we have every primary subcomplex of a complex over von Neumann regular ring is prime.

\begin{theorem}\label{equi}
	Suppose $S$ is a proper subcomplex of an $R-$complex $C$ and $(S_i:C_i)=P_i$ for $i \in P_I$. Then the following assertions are equivalent:
	\begin{enumerate}
		\item $S$ is a prime subcomplex of $C$;
		\item every component $\frac{C_i}{S_i}$ of the quotient complex  $\frac{C}{S}$ is a torsion-free $\frac{R}{P_i}-$module;
		\item for each subcomplex $T$ of $C$ and a family of ideals $\{J_i\}_{i \in I}$ of $R$ with $J_i T_i \subseteq S_i$, $T_i \subseteq S_i$ or $J_i \subseteq (S_i:C_i)$, $\forall i \in P_I$;
		\item $(S_i:r)=S_i$ for all $r \in R - P_i$ and $ i \in P_I$;
		\item $(S_i: J)=S_i$, $\forall i \in P_I$ and any ideal $J$ of $R$;
		\item $(S_i:m)=P_i$, for all $m \in C_i-S_i$ and $ i \in P_I$;
		\item $(S_i:N_i)=P_i$, for all subcomplex $N=(N_i, d_i)$ of $C$ with $S\subset T$ and $ i \in P_I$;
		\item $Ass(\frac{C_i}{S_i}) = \{P_i\}$, $\forall i \in P_I$, where $Ass(\frac{C_i}{S_i})$ indicates the set of associative primes of $\frac{C_i}{S_i}$;
		\item $Z_R(\frac{C_i}{S_i})=P_i$, $\forall i \in P_I$;
		\item every nonzero subcomplex $\frac{T}{S}$ of $\frac{C}{S}$ has the same annihilator, that is, $(T:C)=(S:C)$;
		\item for any family of ideals $J=<J_i>$ of $R$ with $(S_i : C_i)\subsetneq J_i$ and all subcomplex $T$ of $C$ with $S_i \subsetneq T_i$, $J_i T_i \nsubseteq S_i$, $\forall i \in P_I$.
		
	\end{enumerate}
\end{theorem}

\begin{proof}

$(1 \rightarrow 10):$ It suffices to show for any $m+S_i \in \frac{C_i}{S_i}$ with $m\in C_i-S_i$, $Ann(m+S_i)=(S_i:C_i)$. Taking $r \in (S_i: <m+S_i>)$ gives $rm \in S_i$ and we obtain from part (1) that $r \in (S_i:C_i)$.

$(10 \rightarrow 1):$ Suppose $rm \in S_i$, for $i \in P_I$, $m \in C_i \ S_i$ and $r \in R$. Part (10) gives $r \in (S_i:<m+S_i>)=(S_i:C_i)$.

$(11 \rightarrow 3):$ Suppose $J_i T_i \subseteq S_i$ for any subcomplex $T$ of $C$ and a family of ideals $\{J_i\}_{i \in I}$. Set $A_i=S_i+T_i$ and $B_i=J_i+(S_i:C_i)$. Then $T_i \nsubseteq S_i$ if and only if $A_i \nsubseteq S_i$ and $J_i \nsubseteq (S_i:C_i)$ if and only if $B_i \nsubseteq (S_i:C_i)$. Hence $B_i A_i \subseteq S_i$ and we conclude from part (10) that $T_i \subseteq S_i$ or $J_i \subseteq (S_i:C_i)$.

The remains are clear.
\end{proof}

\begin{theorem}\label{p2}
	A subcomplex $S$ of an $R-$complex $C$ with $P_i=(S_i:C_i)$ a maximal ideal of $R$ for all $ i \in P_I$ is a prime subcomplex of $C$.
\end{theorem}
\begin{proof}
	The hypothesis gives us that $\frac{C_i}{S_i}$ is a vector space over the field $\frac{R}{P_i}$ and so a torsion-free $\frac{R}{P_i}-$module. The rest comes from Theorem \ref{equi}.
\end{proof}

\begin{corollary}
	For each maximal ideal $m$ of $R$ with $mC \neq C$, $mC$ is a prime subcomplex of $C$.
\end{corollary}

\begin{corollary}
	Let $\{m_i\}_{i \in I}$ be a family of maximal ideals of $R$. A proper subcomplex $S$ is $m_i-$prime if and only if $m_i C_i \subseteq S_i$. In particular, if $S$ is an $m_i-$prime subcomplex of $C$, then each proper subcomplex $T$ of $C$ containing $S$ is prime.
\end{corollary}

 A proper subcomplex $S$ of an $R-$complex $C$ is said to be a maximal subcomplex if there is no proper subcomplex contains $S$ properly, or equivalently, if $S_i$ is a maximal submodule of $C_i$ for all $i \in P_I$. In commutative ring with identity, every maximal ideal is prime. Now, we prove this for subcomplexes.

\begin{proposition}
	Suppose $S$ is a maximal subcomplex of a complex $C$. Then $S$ is a prime subcomplex of $C$ and $(S_i:C_i)$ is a maximal ideal of $R$.
\end{proposition}

	\begin{proof}
	For $i \in I$, $S_i$ is maximal submodule of $C_i$ if and only if $\frac{C_i}{S_i}$ is a simple $R-$module. Thus $\frac{C_i}{S_i}$ a cyclic $R-$module say $\frac{C_i}{S_i}= Rx_i$, $x_i \in C_i$ and $Ann(Rx_i)= Ann(\frac{C_i}{S_i})= (S_i:C_i)$ is a maximal ideal of $R$ by \cite[Proposition 3]{bour}. Hence $S$ is a prime subcomplex of $C$ due to Proposition \ref{p2}
\end{proof}

A finitely generated complex is a complex whose terms are finitely generated. So we have the following.

\begin{corollary}
	If $C$ is a finitely generated $R-$complex, then every proper subcomplex of $C$ is contained in a prime subcomplex.
\end{corollary}

\vspace{.3cm}
An $R-$complex $F=\{F_i,d_i\}_{i \in I}$ is called flat if all its component are flat over $R$.
\vspace{.3cm}

\begin{lemma}\label{lem}
	Let $C$ and $F$ be $R-$complexes and $S$ be a subcomplex of $C$. Then $F_i \otimes (S_i:r) \subseteq (F_i \otimes S_i:r)$, for any $r \in R$. The equality holds if $F$ is a flat complex.
\end{lemma}
\begin{proof}
	Let $f \otimes c \in F_i \otimes (S_i: r)$ for $f \in F_i$, $rc \in S_i$ and $i \in I$. Then $r (f \otimes c) = f \otimes rc \in F_i \otimes S_i$. Now, suppose $F$ is flat. Then clearly $F_i \otimes \frac{C_i}{S_i} \cong \frac{F_i \otimes C_i}{F_i \otimes S_i}$. Also, exactness of $$0 \to F_i \otimes (S_i:r) \to F_i \otimes C_i  \stackrel{i \otimes f_r}\to \frac{F_i \otimes C_i}{F_i \otimes S_i}$$
	gives that $Im(i \otimes f_r)=F_i \otimes (S_i :r)= (F_i \otimes S_i:r)$, $\forall i \in I$.
\end{proof}

\begin{theorem}\label{tt}
	Let $S$ be a prime subcomplex of $R-$complex $C$ and $F$ be flat $R-$complex with $F \otimes S \neq F \otimes C$. Then $F \otimes S $ is a prime subcomplex of $F \otimes C$.
\end{theorem}
\begin{proof}
	We prove this theorem by the equivalence (4) of prime subcomplexes in Theorem \ref{equi}. For this purpose, let $r \in R- (S_i:C_i)$ for $i \in P_I$. Then the primeness of $S$ and Theorem \ref{equi} (4) gives us $(S_i:r)= S_i$ and we obtain from Lemma \ref{lem} that $(F_i \otimes S_i:r) =F_i \otimes (S_i:r)= F_i \otimes S_i$. Therefore by Theorem \ref{equi} again, we conclude that $F\otimes S$ is a prime subcomplex of $F \otimes C$.
\end{proof}

The converse of Theorem \ref{tt} is true in the case where $F$ is faithfully flat, that is, all components of $F$ are faithfully flat.

\begin{theorem}\label{ff}
	Let $S$ be a subcomplex of $R-$complex $C$ and $F$ be a faithfully flat complex over $R$. Then $S$ is a prime subcomplex of $C$ if and only if $F\otimes S$ is a prime subcomplex of $F \otimes C$.
\end{theorem}
\begin{proof}
	Since the exactness of the sequence
	$$0 \to F_i \otimes S_i \to F_i \otimes C_i \to 0$$ implies the exactness of of the sequence
	$$0 \to S_i \to C_i \to 0,$$
	so we can assume that $F_i \otimes S_i \neq F_i \otimes C_i$, $\forall i \in P_I$ and Theorem \ref{tt} tells us that $F \otimes S $ is a prime subcomplex of $F\otimes C$. Now, let $F \otimes S$ be a prime subcomplex of $F \otimes C$ and so $S$ is a proper subcomplex of $C$. Thus for any $r \in R - (S_i: C_i)$, $\forall i \in P_I$, we have $r \in R - (F_i \otimes S_i: F_i \otimes C_i)$ and Lemma \ref{lem} implies that $$F_i \otimes (S_i:r) = (F_i \otimes S_i:r)= F_i \otimes S_i$$
	 As $F$ is faithfully flat complex, we have $(S_i:r)=S_i$, $\forall i \in P_I$. Hence by Theorem \ref{equi} (4) again, $S$ is a prime subcomplex of $C$ as desired.
\end{proof}

\begin{corollary}

	\begin{enumerate}
		\item Let $T$ be a multiplicative subset of $R$ and $S = (S_i,d_i)$ be a proper subcomplex of an $R-$complex $C=(C_i, d_i)$ with $T \cap (S_i:C_i)= \phi$, $\forall i \in P_I$. Then $T^{-1}S = (T^{-1}S_i, T^{-1}d_i)$ is a prime subcomplex of $T^{-1}C= (T^{-1}C_i, T^{-1}d_i)$ if and only if $S$ is a prime subcomplex of $C$.
		
		\item Let $T= \{T_i: i \in I\}$ be a family of multiplicative subsets of $R$ and $S=(S_i,d_i)$ be a proper subcomplex of $C=(C_i,d_i)$ such that $(S_i:C_i) \cap T_i = \phi$, $\forall i \in P_I$. Then $T^{-1}S=(T_i^{-1} S_i, T^{-1}_i d_i)$ is a prime subcomplex of $T^{-1}C= (T_i^{-1}C_i, T_i^{-1}d_i)$ if and only if $S$ is a prime subcomplex of $C$.
	\end{enumerate}	
\end{corollary}
\begin{proof}
	The proof come from Theorem \ref{ff} and the fact that $T^{-1}S=T^{-1}R \otimes S$.
\end{proof}

\begin{corollary}
	Let $S=(S_i,d_i)$ be a prime subcomplex of an $R-$complex $C=(C_i,d_i)$ and $x$ be an indeterminate. Then $S[x]=(S_i[x], d^*_i)$ is a prime subcomplex of $C[x]= (C_i[x], d^*_i)$, where $d_i^*(\sum_{i=1} ^n m_i x^i) = \sum_{i=1} ^n d_i(m_i) x^i$, for $m_i \in S_i$.
\end{corollary}
\begin{proof}
	The proof follows from the fact that $S[x]=S \otimes R[x]$ and Theorem \ref{ff}.
\end{proof}

We now generalize the prime avoidance theorem to prime subcomplexes.

\begin{theorem}\label{pri}
	Let $T_1, \cdots , T_n$ be subcomplexes of an $R-$complex $C$ and let $S$ be a prime subcomplex of $C$ with $T_1 \cap T_2 \cap  \cdots \cap T_n \subseteq S$. Then there exists an $i \in \{1, \cdots , n\}$ such that either $T_i \subseteq S$ or $(T_i :C) \subseteq (S:C)$.
\end{theorem}

\begin{proof}
	Assume by contrary that $T_1 \nsubseteq S$ and $((T_i)_j:C_j) \nsubseteq (S_j:C_j)$ for some $j \in P_I$. So, there exist $m \in (T_1)_j - S_j$ and $r_i \in ((T_i)_j:C_j) - (S_j:C_j)$ for all $i=1, \cdots , n$. Thus $r_i m \in (T_1)_j \cap (T_i)_j$ for $i \in \{2, \cdots , n\}$, so $r_1 r_2 \cdots r_n m \in (T_1)_j \cap (T_2)_j \cap \cdots \cap (T_n)_j= (T_1 \cap \cdots \cap T_n)_j \subseteq S_j$. The primeness of $S$ gives us $m \in S_j$ or $r_2 r_3 \cdots r_n \in (S_j: C_j)$ which contradicts our assumption.
\end{proof}

At the end of the paper I want to mention that there are a lot of open questions on such subject. For example, one can do research on how to generalize the supplemented modules and its generalizations into the category of $R-$complexes.

\vspace{.3cm}

\textbf{ The author declare that he has no conflict of interest.}

\vspace{.3cm}

\textbf{This manuscript does not include any data.}
\vskip 0.4 true cm

\end{document}